\documentclass[coll]{imsart}

\usepackage{amsthm,amsmath,amsfonts,natbib}
\usepackage{url}
\RequirePackage[colorlinks,citecolor=blue,urlcolor=blue]{hyperref}

\RequirePackage{hypernat}

\startlocaldefs

\newtheorem{corollary}{Corollary}
\newtheorem{lemma}{Lemma}
\newtheorem{proposition}{Proposition}

\theoremstyle{remark}
\newtheorem{example}{Example}

\newfont{\msbm}{msbm10 at 11pt}
\newcommand {\sZ} {\mathsf{Z}}
\newcommand {\cB} {{\cal B}}
\newcommand {\sX} {\mathsf{X}} 
\newcommand {\sY} {\mathsf{Y}}

\def\baro{\vskip  .2truecm\hfill \hrule height.5pt \vskip  .2truecm}
\def\barba{\vskip -.1truecm\hfill \hrule height.5pt \vskip .4truecm}
\endlocaldefs
\begin{document}

\begin{frontmatter}

\title{On the Geometric Ergodicity of Two-Variable Gibbs Samplers}
\runtitle{Geometric Ergodicity}

\author{\fnms{Aixin} \snm{Tan}\corref{}\ead[label=e1]{aixin-tan@uiowa.edu}}
\address{Department of Statistics and Actuarial Science 
University of Iowa 
Iowa City, IA 52242
\printead{e1}}
\affiliation{University of Iowa}

\author{\fnms{Galin L.} \snm{Jones}\ead[label=e2]{galin@umn.edu}}
\address{School of Statistics 
University of Minnesota 
Minneapolis, MN 55455 
\printead{e2}}

\affiliation{University of Minnesota}

\author{\fnms{James P.} \snm{Hobert}\ead[label=e3]{jhobert@stat.ufl.edu}}
\address{Department of Statistics 
University of Florida 
Gainesville, FL 32611
\printead{e3}}

\affiliation{University of Florida}

\thankstext{t1}{Jones is partially supported by the National
  Institutes of Health.  Hobert is partially supported by the National
  Science Foundation.}  
\runauthor{Tan et al.}

\begin{abstract}
  A Markov chain is geometrically ergodic if it converges to its
  invariant distribution at a geometric rate in total variation norm.
  We study geometric ergodicity of deterministic and random scan
  versions of the two-variable Gibbs sampler.  We give a sufficient
  condition which simultaneously guarantees both versions are
  geometrically ergodic.  We also develop a method for simultaneously
  establishing that both versions are subgeometrically ergodic.  These
  general results allow us to characterize the convergence rate of
  two-variable Gibbs samplers in a particular family of discrete
  bivariate distributions.
\end{abstract}



\end{frontmatter}

\section{Introduction}
\label{sec:introduction}
Let $\varpi$ be a probability distribution having support $\sX \times
\sY \subseteq \mathbb{R}^{k} \times \mathbb{R}^{l}$, $k, l \ge 1$ and
$\varpi_{X|Y}$ and $\varpi_{Y|X}$ denote the associated conditional
distributions.  We assume it is possible to simulate directly from
$\varpi_{X|Y}$ and $\varpi_{Y|X}$. Then there are two Markov chains
having $\varpi$ as their invariant distribution, each of which could
be called a \textit{two-variable Gibbs sampler} (TGS). The most common
version of a TGS is the deterministic scan Gibbs sampler (DGS), which
is now described. Suppose the current state of the chain is $(X_{n},
Y_{n}) = (x, y)$, then the next state, $(X_{n+1}, Y_{n+1})$, is
obtained as follows.

\baro
Iteration $n+1$ of DGS:
\begin{enumerate}
\item Draw $X_{n+1} \sim \varpi_{X|Y} ( \cdot | y)$, and call the
  observed value $x'$.
\item Draw $Y_{n+1} \sim \varpi_{Y|X} (\cdot | x')$.
\end{enumerate}
\barba

An alternative TGS is the random scan Gibbs sampler (RGS).  Fix $p
\in  (0,1)$ and suppose the current state of the RGS chain is $(X_{n}, Y_{n}) =
(x, y)$.  Then the next state, $(X_{n+1}, Y_{n+1})$, is obtained as
follows.

\baro
Iteration $n+1$ of RGS:
\begin{enumerate}
\item Draw $B \sim \text{Bernoulli}(p)$.
\item If $B=1$, then draw $X_{n+1} \sim \varpi_{X|Y} ( \cdot | y)$ and
  set $Y_{n+1}=y$.
\item If $B=0$, then draw $Y_{n+1} \sim \varpi_{Y|X} (\cdot | x)$ and
  set $X_{n+1} = x$.
\end{enumerate}
\barba

Despite the simple structure of either TGS, these algorithms are
widely applicable in the posterior analysis of complex Bayesian
models.  A TGS also arises naturally when $\varpi$ is created via
data augmentation techniques \citep{hobe:2011,tann:wong:1987}.

Inference based on $\varpi$ often requires calculation of an
intractable expectation.  Let $g : \sX \times \sY \to \mathbb{R}$ and
let $E_{\varpi} g$ denote the expectation of $g$ with respect to
$\varpi$. If a TGS Markov chain is ergodic 
\citep[see][]{tier:1994} 
and $E_{\varpi} |g| < \infty$, then 
\[
\bar{g}_{n} : = \frac{1}{n} \sum_{i=0}^{n-1} g(X_{i}, Y_{i})
\stackrel{a.s.}{\longrightarrow} E_{\varpi} g \; \qquad \text{ as } n \to \infty .
\]
Thus estimation of $E_{\varpi} g$ is simple.  However, the estimator
$\bar{g}_{n}$ will be more valuable if we can attach an estimate of
the unknown Monte Carlo error $\bar{g}_{n} - E_{\varpi} g$.  An
approximation to the sampling distribution of the Monte Carlo error is
available when a Markov chain central limit theorem (CLT) holds
\[
\sqrt{n} (\bar{g}_{n} - E_{\varpi} g) \stackrel{d}{\to} \text{N}(0,
\sigma_{g}^{2}) \qquad \text{ as } n \to \infty 
\]
with $0 < \sigma_{g}^{2} < \infty$.  The variance $\sigma_{g}^{2}$
accounts for the serial dependence in a TGS Markov chain and
consistent estimation of it requires specialized techniques such as
batch means, spectral methods or regenerative simulation.  Let
$\hat{\sigma}_{n}^{2}$ be an estimator of $\sigma_{g}^{2}$.  If, with
probability 1, $\hat{\sigma}_{n}^{2} \to \sigma_{g}^{2}$ as $n \to
\infty$, then an asymptotically valid Monte Carlo standard error is
$\hat{\sigma}_{n} / \sqrt{n}$.  These tools allow the practitioner to
use the results of a TGS simulation with the same level of confidence
that one would have if the observations were a random sample from
$\varpi$.  For more on this approach the interested reader can consult
\citet{geye:1992}, \citet{geye:2011}, \citet{fleg:hara:jone:2008},
\citet{fleg:jone:2010}, \citet{fleg:jone:2011},
\citet{hobe:jone:pres:rose:2002}, \citet{jone:hara:caff:neat:2006},
and \citet{jone:hobe:2001}.

The CLT will obtain if $E_{\varpi} |g|^{2+\epsilon} < \infty$ for some
$\epsilon > 0$ and the Markov chain is rapidly mixing
\citep{chan:geye:1994}.  In particular, we require that the Markov
chain be geometrically ergodic; that is, converge to the target
$\varpi$ in total variation norm at a geometric rate. Under these same
conditions methods such as batch means and regenerative simulation
provide strongly consistent estimators of $\sigma^{2}_{g}$. Thus
establishing geometric ergodicity is a key step in ensuring the
reliability of a TGS as a method for estimating features of $\varpi$.

The convergence rate of DGS Markov chains has received substantial
attention.  In particular, sufficient conditions for geometric
ergodicity have been developed for several DGS chains for practically
relevant statistical models \citep[see
e.g.][]{hobe:geye:1998,john:jone:2010,jone:hobe:2004,marc:hobe:2004,robe:pols:1994,robe:rose:1999a,roma:hobe:2011,roma:2012,rose:1996,roy:hobe:2007,tan:hobe:2009}. The
convergence rates of RGS chains has received almost no attention
despite sometimes being useful.  \citet{liu:wong:kong:1995} did
investigate geometric convergence of RGS chains, but the required
regularity conditions are daunting and, to our knowledge, have not
been applied to practically relevant statistical models.  Recently
\citet{john:jone:neat:2011} gave conditions which simultaneously
establish geometric ergodicity of both the DGS chain and the
corresponding RGS chain.  These authors also conjectured that if the
RGS chain is geometrically ergodic, then so is the DGS chain.  That is
to say, the qualitative convergence properties of TGS chains coincide.
We are not able to resolve this conjecture in general, but in our main
application (see Section~\ref{sec:example}) this is indeed the case.

A TGS chain which converges subgeometrically (ie, slower than
geometric) would not be as useful as another chain which is
geometrically ergodic--although with additional moment conditions it
is still possible for a CLT to hold \citep{jone:2004}.  Thus it would
be useful to have criteria to check for subgeometric convergence.  We
are unaware of any previous work investigating subgeometric
convergence of TGS Markov chains.

In the rest of this paper, we extend the results of
\citet{john:jone:neat:2011} and provide a condition which can be used
to simultaneously establish geometric ergodicity of DGS and RGS Markov
chains.  We then turn our attention to development of a condition
which ensures that both the DGS and RGS chains converge
subgeometrically.  Finally, we apply our results to a class of
bivariate distributions where we are able to characterize the
convergence properties of the DGS and RGS chains.  But we begin with
some Markov chain background and a formal definition of the Markov
chains we study.

\section{Background and Notation}
\label{sec:background}

Let $\sZ$ be a topological space and $\cB(\sZ)$ denote its Borel
$\sigma$-algebra. Also, let $\Phi = \left\{Z_{0}, Z_{1}, Z_{2}, \ldots
\right\}$ be a Markov chain having Markov transition kernel $P$.  That
is, $P : \sZ \times \cB(\sZ) \to [0,1]$ such that for each $A \in
\cB(\sZ)$, $P(\cdot, A)$ is a nonnegative measurable function and for
each $z \in \sZ$, $P(z, \cdot)$ is a probability measure. As usual,
$P$ acts to the left on measures so that if $\nu$ is a measure on
$(\sZ, \cB(\sZ))$ and $A \in \cB(\sZ)$, then
\[
\nu P (A) = \int_{\sZ} \nu(dz) P(z, A) \; .
\]
For any $n \in \mathbb{Z}^+$, the $n$-step Markov transition kernel is given by 
$P^n(z,A) = \mathrm{Pr}(Z_{n+j} \in A | Z_{j} = z)$.

Let $w$ be an invariant probability measure for $P$, that is,
$w P = w$.  Denote total variation norm by $\|\cdot\|_{TV}$.
If $\Phi$ is ergodic, 
then for all $z \in \sZ$ we have
$||P^n(z, \cdot) - w(\cdot) ||_{TV} \to 0$ as $n \to \infty$.
Our goal is to study the rate of this convergence.  Suppose there
exist a real-valued function $M(z)$ on $\sZ$ and $0 < t < 1$ such that
for all $z$
\begin{equation}
\label{eq:geometric ergodicity}
||P^n(z, \cdot) - w(\cdot) ||_{TV} \leq M(z) t^{n} \; .
\end{equation}
Then $\Phi$ is \textit{geometrically ergodic}, otherwise it is
\textit{subgeometrically ergodic}.

\subsection{Two-variable Gibbs samplers}
\label{sec:three chains}

In this section we define the Markov kernels associated with the DGS
and RGS chains described in Section~\ref{sec:introduction}.  We also
introduce a third Markov chain which will prove crucial to our study
of the other Markov chains.

Recall that $\varpi$ is a probability distribution having support $\sX
\times \sY \subseteq \mathbb{R}^{k} \times \mathbb{R}^{l}$, $k, l \ge
1$. Let $\pi(x, y)$ be a density of $\varpi$ with respect to a measure
$\mu = \mu_{X} \times \mu_{Y}$.  Then the marginal densities are given
by
\[
\pi_{X} (x) = \int_{\sY} \pi(x, y) \mu_{Y}(dy)
\]
and similarly for $\pi_{Y}(y)$. The conditional densities are
$\pi_{X|Y} (x|y) = \pi(x,y) / \pi_{Y}(y)$ and $\pi_{Y|X}(y|x) =
\pi(x,y) / \pi_{X}(x)$.

Consider the DGS Markov chain $\Phi_{DGS} = \{ (X_{0}, Y_{0}), (X_{1}, Y_{1}),
\ldots \}$ and let 
\[
k_{DGS}(x', y' | x, y) = \pi_{X|Y}(x' | y) \pi_{Y|X}(y'|x') \; .
\]
Then the Markov kernel for $\Phi_{DGS}$ is defined by
\[
P_{DGS} ((x,y), A) = \int_{A} k_{DGS}(x', y' | x, y)
\mu(d(x',y'))\qquad A \in \cB(\sX) \times \cB(\sY) \;. 
\]
It is well known that the two marginal sequences comprising $\Phi_{DGS}$ are themselves Markov chains \citep{liu:wong:kong:1994}. We now consider the marginal sequence $\Phi_{X} = \{ X_{0}, X_{1}, 
\ldots \}$ and define
\[
k_{X}(x' | x) = \int_{\sY} \pi_{X|Y}(x' | y) \pi_{Y|X}(y|x)
\mu_{Y}(dy) \; .
\]
The Markov kernel for $\Phi_{X}$ is
\[
P_{X} (x, A) = \int_{A} k_{X}(x' | x) \mu_{X}(dx') \qquad A \in
\cB(\sX) \; .
\]
Note that $P_{DGS}$ has $\varpi$ as its invariant distribution while
$P_{X}$ has the marginal $\varpi_{X}$ as its invariant distribution.

Finally, consider the RGS Markov chain $\Phi_{RGS} = \{ (X_{0},
Y_{0}), (X_{1}, Y_{1}), \ldots \}$.  Let $p \in (0,1)$ and
$\delta$ denote Dirac's delta.  Define
\[
k_{RGS}(x', y' | x, y)= p \pi_{X|Y}(x'|y) \delta(y' - y) + (1-p)
\pi_{Y|X} (y'|x) \delta(x' - x) \; .
\]
Then the Markov kernel for $\Phi_{RGS}$ is
\begin{align*}
P_{RGS}((x,y), A) & = \int_{A} k_{RGS}(x' , y' | x, y) \mu(d(x', y')) 
\end{align*}
It is easy to show via direct computation that $\varpi$ is invariant
for $P_{RGS}$.

It is well known that $P_{X}$ and $P_{DGS}$ converge to their
respective invariant distributions at the same rate
\citep{diac:etal:2008,liu:wong:kong:1994,robe:1995,robe:rose:2001}.
Thus if one is geometrically ergodic, then so is the other. This
relationship has been routinely exploited in the study of TGS chains
for practically relevant statistical models
\citep[cf.][]{hobe:geye:1998,john:jone:2010,jone:hobe:2004,roy:hobe:2007,tan:hobe:2009}
since one of the two chains may be easier to analyze than the other.
Recently, \citet{john:jone:neat:2011} connected the geometric
ergodicity of $P_{X}$ to that of $P_{RGS}$.  Thus establishing
geometric ergodicity of TGS algorithms often comes down to analyzing
$P_{X}$. This is exactly the approach we take in
Sections~\ref{sec:geometric ergodicity} and~\ref{sec:example}.
 
\section{Conditions for Geometric Ergodicity}
\label{sec:geometric ergodicity}

In this section we develop general conditions which ensure that
$P_{X}$, $P_{DGS}$ and $P_{RGS}$ are geometrically ergodic.  First we
need a couple of concepts from Markov chain theory.  Recall the
notation from Section~\ref{sec:background}.
That is, $P$ is a Markov
kernel on $(\sZ, \cB(\sZ))$.  Then $P$ is \textit{Feller}
if for any open set $O \in \cB(\sZ)$, $P(\cdot, O)$ is a lower
semicontinuous function.  The Markov kernel $P$ acts to the right on
functions so that for measurable $f$
\[
Pf(z) = \int_{\sZ} f(z') P(z, dz')\; .
\]
A \textit{drift condition} holds if there exists a function
$U: \sZ \to \mathbb{R}^{+}$, and constants $0 < \lambda < 1$ and $L <
\infty$ satisfying
\begin{equation}
\label{eq:drift}
PU(z)  \le \lambda U(z) + L~~~~~\text{for all } z \in \sZ \,.  
\end{equation}
Recall that a function $U$ is said to be \textit{unbounded off compact
  sets} if the sublevel set $\{z \in \sZ : U(z) \leq d \}$ is compact
for every $d>0$.  If $P$ is Feller, $U$ is unbounded off compact sets
and satisfies \eqref{eq:drift}, then $\Phi$ is geometrically
ergodic. See \citet{meyn:twee:1993} and \citet{robe:rose:2004} for
details while \citet{jone:hobe:2001} give an introduction to the use
of drift conditions.

\subsection{Two-variable Gibbs samplers}

\citet{john:jone:neat:2011} gave a set of conditions which
simultaneously prove that $\Phi_{X}$, $\Phi_{DGS}$ and $\Phi_{RGS}$
are geometrically ergodic.  We build on their work and show how a
drift condition for $P_{X}$ naturally provides a drift condition for
$P_{RGS}$.  This allows us to develop an alternative set of conditions
which are sufficient for the geometric ergodicity of $P_{X}$,
$P_{DGS}$ and $P_{RGS}$.  The application of this method is
illustrated in Section~\ref{sec:example}.

The following result was essentially proved by
\citet{john:jone:neat:2011}, but it was not stated in their paper; see
\citet{john:2009} for related material.  Thus we provide a proof for
the sake of completeness.  First we set some notation.  Suppose $V :
\sX \to \mathbb{R}^{+}$ and let
\[
G(y) = \int_{\sX} V(x) \pi_{X|Y}(x|y) \mu_{X}(dx) \; .
\]
Also, for $c > 0$ define 
\begin{equation}
\label{eq:drift fn}
W(x,y) = V(x) + cG(y) \; .
\end{equation}

\begin{lemma} \label{lem:drift} 
Suppose there exist constants $0 < \lambda < 1$ and $L < \infty$ such
that for all $x \in \sX$ 
\[
P_{X} V(x) \le \lambda V(x) + L \; .
\]
If $0 < p < 1$ and $p(1-p)^{-1} < c < p[\lambda (1-p)]^{-1}$, then
there exists $\lambda < \gamma < 1$ such that
\[
P_{RGS} W(x,y) \le \gamma W(x,y) + (1-p)cL \; .
\]
\end{lemma}

\begin{proof}
Notice that
\begin{align*}
\int_{\sY} G(y) \pi_{Y|X} (y|x) \mu_{Y}(dy) & = \int_{\sY} \int_{\sX}
V(x') \pi_{X|Y}(x'|y) \pi_{Y|X} (y|x) \mu_{X}(dx') \mu_{Y}(dy)\\
& = \int_{\sX} V(x') \int_{\sY} \pi_{X|Y}(x'|y) \pi_{Y|X}(y|x)
\mu_{Y}(dy) \mu_{X}(dx')\\
& = \int_{\sX} V(x') k_{X}(x' | x) \mu_{X}(dx')\\
& \le \lambda V(x) + L \; .
\end{align*}
Since 
\begin{equation} \label{eq:c values}
\frac{p}{1-p} < c < \frac{p}{\lambda(1-p)} 
\end{equation}
there exists $\gamma$ such that
\begin{equation} \label{eq:gamma values}
(1-p)( c \lambda + 1)\vee \frac{p(1+c)}{c} \le \gamma < 1 \, .
\end{equation}
Then
\begin{align*}
  P_{RGS} W(x,y) & = \int_{\sX} \int_{\sY} W(x' , y') k_{RGS} (x' , y'
  | x, y) \mu_{X} (dx') \mu_{Y}(dy')\\
  & = p \int_{\sX} \int_{\sY} W(x', y') \pi_{X|Y}(x' | y) \delta(y'-y)
  \mu_{X} (dx') \mu_{Y}(dy') \\
  & \qquad + (1-p) \int_{\sX} \int_{\sY} W(x', y') \pi_{Y|X}(y'|x)
  \delta(x' - x) \mu_{X} (dx') \mu_{Y}(dy') \\
  & = p \int_{\sX}   W(x', y) \pi_{X|Y}(x' | y)  \mu_{X}(dx') + (1-p)
  \int_{\sY} W(x, y') \pi_{Y|X} (y'|x)  \mu_{Y} (dy') \\
 & = p  \int_{\sX} \left[ V(x') + c G(y) \right] \pi_{X|Y}(x' | y)
 \mu_{X}(dx') \\
& \qquad + (1-p) \int_{\sY} \left[V(x) + c G(y') \right]
\pi_{Y|X}(y'|x) \mu_{Y}(dy') \\ 
& = pcG(y) + (1-p) V(x) + p G(y) + (1-p) c \int_{\sY} G(y') \pi_{Y|X}
(y' | x) \mu_{Y} (dy') \\
& = p(1+c) G(y) + (1-p) V(x) + (1-p) c \int_{\sY} G(y') \pi_{Y|X} (y'
| x) \mu_{Y} (dy') \\
& \le (1-p) c \lambda V(x) + (1-p) c L + p(1+c) G(y) + (1-p) V(x) \\
& = (1-p) ( c \lambda + 1) V(x) + p(1+c) G(y) + (1-p)cL \\
& \le \gamma W(x,y) + (1-p) cL \; .
\end{align*}
All that remains is to show that $\gamma > \lambda$. Now
\begin{align*}
\gamma &\ge (1-p) ( c \lambda + 1) \qquad \qquad \text{ by } \eqref{eq:gamma
values}\\
 & > (1-p) \left( \frac{p}{1-p} \lambda + 1 \right) ~\quad \text{ by }
 \eqref{eq:c values} \\
 & = p \lambda + (1-p) \\
 & > \lambda \qquad \qquad \qquad \qquad \qquad\text{ since } \lambda,
 \, p \in (0,1) \; . 
\end{align*}
\end{proof}

The following is an easy consequence of Lemma~\ref{lem:drift} and the
material stated at the beginning of this section.

\begin{proposition} \label{prop:feller drift} Suppose $P_{X}$ and
  $P_{RGS}$ are Feller. If there exists a function $V: \sX \to
  \mathbb{R}^{+}$ such that both $V$ and the corresponding $W$ (as
  defined at \eqref{eq:drift fn}) are unbounded off compact sets, and
  there exist constants $0 < \lambda < 1$ and $L < \infty$ such that
  for all $x \in \sX$
\[
P_{X} V(x) \le \lambda V(x) + L,
\]
then $\Phi_{X}$, $\Phi_{DGS}$ and $\Phi_{RGS}$ are geometrically ergodic.
\end{proposition}

\section{Conditions for Subgeometric Convergence}
\label{sec:subgeometric convergence}
 
Our goal in this section is to develop a condition which ensures that
$\Phi_{X}$, $\Phi_{DGS}$ and $\Phi_{RGS}$ converge subgeometrically,
but first we need a few concepts from general Markov chain theory.
Recall the notation of Section~\ref{sec:background}. In particular,
$P$ is a Markov kernel on $(\sZ, \cB(\sZ))$ having invariant
distribution $w$.  A Markov kernel defines an operator on the space of
measurable functions that are square integrable with respect to the
invariant distribution, denoted $L^{2}(w)$. Also, let
\[
L_{0,1}^{2}(w) = \left\{ f \in L^{2}(w) \, : \, E_{w} f = 0,
  \text{ and } E_{w} f^{2} = 1 \right\} \; .
\]
For $f, g \in L^{2}(w)$, define the inner product as
\[
\langle f, g \rangle = \int_{\sZ} f(z) g(z) w(dz) 
\]
and $\|f\|^{2} = \langle f, f \rangle$.  The norm of the operator $P$
is 
\[
\| P \| = \sup_{f \in L_{0,1}^{2}(w)} \|Pf \| \; .
\]
If $P$ is symmetric with respect to $w$, that is, if
\begin{equation}
\label{eq:symmetry}
P(z, dz')w(dz) = P(z', dz) w(dz'), 
\end{equation}
then $P$ is self-adjoint so that $\langle P h_{1}, h_{2} \rangle = \langle
h_{1}, P h_{2} \rangle$.  If $P$ is $w$-symmetric, then $\Phi$ is
geometrically ergodic if and only if $\|P\| < 1$ \citep{robe:rose:1997c}. 
Moreover, if $
Z\sim w$ and $Z'|Z=z\sim P(z, \cdot)$, then 
\begin{equation}
\label{eq:operator norm}
\|P \|  = \sup_{f \in
L_{0,1}^{2}(w)} |\langle P f, f \rangle | = \sup_{f \in
L_{0,1}^{2}(w)} |E \left[ f(Z') f(Z) \right] | \; .
\end{equation}
The first equality is a property of self-adjoint operators while the
second equality follows directly from the definition of inner product.

\subsection{Two-variable Gibbs samplers}

It is easy to see that $P_{X}$ is $\varpi_{X}$-symmetric and $P_{RGS}$
is $\varpi$-symmetric, but $P_{DGS}$ is not $\varpi$-symmetric.
Because $P_{X}$ and $P_{RGS}$ are symmetric, the operator theory
described above applies.  In particular, if $X \sim \varpi_{X}$ and $X'|X=x
\sim P_{X}(x, \cdot)$, then
\[
\| P_{X} \| = \sup_{f \in
L_{0,1}^{2}(\varpi_{X})} | E [f(X') f(X)] |
\]
while if $(X, Y) \sim \varpi$ and $(X', Y')| (X, Y)=(x, y) \sim P_{RGS}((x,y), \cdot)$, then
\[
\| P_{RGS} \|  = \sup_{f \in
L_{0,1}^{2}(\varpi) } | E [f(X',Y') f(X, Y)] |\; .
\]
Note that despite our use of $\| \cdot \|$ for both operator norms, these
are different since they are based on different $L^{2}$ spaces. 

If we can show that $\|P_{X}\|=\|P_{RGS}\|=1$, then we will be able to
conclude that $\Phi_{X}$, $\Phi_{DGS}$, and $\Phi_{RGS}$ are
subgeometrically ergodic. First, we need convenient characterizations
of the operator norms.

\begin{lemma} \label{lem:operator norms}
If $(X, Y) \sim \varpi$, then
\begin{align*}
\|P_{X} \| & = 1 - \inf_{f \in L_{0,1}^{2}(\varpi_{X})} E(\text{Var}(f(X)|Y))
\end{align*}
and 
\begin{align*}
\|P_{RGS} \| & = 1 - \inf_{f \in L_{0,1}^{2}(\varpi)} \left\{p E (
  \text{Var} ( f(X, Y) | Y)) + (1-p) E ( \text{Var} (f(X, Y) | X)) \right\} \; .
\end{align*}
\end{lemma} 

\begin{proof}
Suppose $X \sim \varpi_{X}$, $X'|X=x \sim P_{X}(x, \cdot)$   and $(X, Y)
\sim \varpi$.  Then 
\begin{align*}
\|P_{X} \| & = \sup_{f \in L_{0,1}^{2}(\varpi_{X})} | E [f(X') f(X)] | \\
& = \sup_{f \in L_{0,1}^{2}(\varpi_{X})}\text{Var}(E (f(X) | Y)) \\
& = 1 - \inf_{f \in L_{0,1}^{2}(\varpi_{X})} E(\text{Var}(f(X)|Y)) \; .
\end{align*}
In the above, the second equality follows from \citet[Lemma
3.2]{liu:wong:kong:1994} and the last equality holds since for $f \in L_{0,1}^{2}(\varpi_{X})$
\[
1 = E(\text{Var}(f(X) | Y)) + \text{Var}(E(f(X)|Y))\,.
\]

Now consider $\|P_{RGS}\|$.  Suppose $(X,Y) \sim \varpi$ and $(X',Y')|(X,Y)=(x,y)
\sim P_{RGS}((x,y),\cdot)$.  Then

\begin{align*}
& E \left[ h(X', Y') h(X, Y) \right] \\ & = \int h(x',y')
h(x,y) k_{RGS}(x',y' | x, y) \pi(x,y) \mu_{X}(dx') \mu_{Y}(dy') \mu_{X}(dx)
\mu_{Y}(dy)\\
& = \int  h(x',y') h(x,y)\pi(x,y) [ p \pi_{X|Y}(x' | y) \delta(y' - y) \\
& \qquad + (1-p) \pi_{Y|X}(y' | x) \delta(x' - x)] \mu_{X}(dx') \mu_{Y}(dy') \mu_{X}(dx)
\mu_{Y}(dy)\\
& = \int p h(x',y) h(x,y) \pi_{X|Y}(x'|y) \pi(x,y) \mu_{X}(dx') \mu_{X}(dx)
\mu_{Y}(dy)\\
& \qquad  + \int (1-p)  h(x,y') h(x,y)\pi_{Y|X}(y' | x)\pi(x,y) \mu_{Y}(dy') \mu_{X}(dx)
\mu_{Y}(dy)\\
& = \int p h(x,y)  E[ h(X', Y) |Y=y]  \pi(x,y)  \mu_{X}(dx) \mu_{Y}(dy)\\
& \qquad  + \int (1-p)h(x,y)  E[ h(X, Y') | X=x] \pi(x,y) \mu_{X}(dx) \mu_{Y}(dy)\\
\end{align*}
\begin{align*}
& = \int p h(x,y)  E[ h(X', Y) |Y=y]  \pi_{X|Y}(x|y) \pi_{Y}(y) \mu_{X}(dx) \mu_{Y}(dy)\\
& \qquad  + \int (1-p)h(x,y)  E[ h(X, Y') | X=x] \pi_{Y|X}(y|x)
\pi_{X}(x) \mu_{X}(dx) \mu_{Y}(dy)\\ 
& = \int p E[ h(X, Y) | Y=y]   E[ h(X',Y) |Y=y]  \pi_{Y}(y) \mu_{Y}(dy) \\
& \qquad  + \int (1-p) E[ h(X, Y) | X=x]  E[ h(X, Y') | X=x] \pi_{X}(x)
\mu_{X}(dx) \\
& = \int p  (E[ h(X,Y) |Y=y])^{2}  \pi_{Y}(y) \mu_{Y}(dy)\\
& \qquad  + \int (1-p)  (E[ h(X, Y) | X=x])^{2} \pi_{X}(x) \mu_{X}(dx)\\
& = p E \left[ \left( E \left[ h(X, Y) | Y \right]\right)^{2} \right]
+ (1-p) E \left[ \left( E \left[ h(X, Y) | X \right]\right)^{2}
\right]\,. 
\end{align*}
Now since $h \in L_{0,1}^{2}(\varpi)$\,,
\[
\text{Var} ( E[ h(X, Y) | Y] ) = E [\left( E \left[ h(X, Y) | Y \right]\right)^{2}]
\]
and 
\[
 \text{Var} ( E [ h(X, Y) | X] ) =  E [\left( E \left[ h(X, Y) | X \right]\right)^{2}] \, .
\]
Moreover,
\[
1 =  \text{Var}_{\varpi} [ h(X, Y) ] =  \text{Var} ( E[ h(X, Y) | Y] )
+ E ( \text{Var} [ h (X, Y) | Y]) 
\]
and
\[
1 =  \text{Var}_{\varpi} [ h(X, Y) ] =  \text{Var} ( E[ h(X, Y) | X] )
+ E ( \text{Var} [ h (X, Y) | X]) \; .
\]
The result follows easily. 
\end{proof}

\begin{proposition} \label{prop:notGE}
Suppose there exists a sequence $\{ h_{i} \in L_{0,1}^{2}
(\varpi_{X})\}$ such that if $(X, Y) \sim \varpi$, then
\begin{equation}
\label{eq:liminf}
\liminf_{i \to \infty} E[ Var(h_{i}(X) | Y)] = 0 \; .
\end{equation}
Then $\|P_{X} \| = \|P_{RGS}\| = 1$.  Hence $\Phi_{X}$, $\Phi_{RGS}$
and $\Phi_{DGS}$ are subgeometrically ergodic.
\end{proposition}

\begin{proof}
The claim that $\|P_{X} \|= 1$ follows easily from the first part of
Lemma~\ref{lem:operator norms}.  Now consider $\|P_{RGS} \|$.  Note that
  if $f'(x,y) := f(x) \in L_{0,1}^{2}(\varpi_{X})$, then $f' \in L_{0,1}^{2}(\varpi)$.
  From the second part of Lemma~\ref{lem:operator norms} we have
\[
\|P_{RGS} \|= 1 - \inf_{f \in L_{0,1}^{2}(\varpi)} \left\{p E (
  \text{Var} [ f(X, Y) | Y])+ (1-p) E ( \text{Var} [ f(X, Y) | X])
\right\} \; .
\]
The claim now follows easily since if $f(x,y) = h_{i}(x)$, then
\[
E ( \text{Var} [ f(X, Y) | X]) = E ( \text{Var} [ h_{i}(X) | X]) = 0
\]
and 
\[
E ( \text{Var} [ f(X, Y) | Y]) = E ( \text{Var} [ h_{i}(X) | Y] ) \; .
\]
Thus we conclude that $\Phi_{X}$ and $\Phi_{RGS}$ are
subgeometrically ergodic.  Since $\Phi_{X}$ and $\Phi_{DGS}$ are either
both geometrically ergodic or both subgeometric, it follows that
$\Phi_{DGS}$ also converges subgeometrically.
\end{proof}

\section{A Discrete Example}
\label{sec:example}

We introduce a family of simple discrete distributions which admit
usage of the TGS algorithms.  We then apply our general results which
will allow us to very nearly characterize the members of the family
which admit geometrically ergodic TGS Markov chains.

Let $\{ a_{i} \}_{i=1}^{\infty}$ and $\{b_{i}\}_{i=1}^{\infty}$ be
strictly positive sequences satisfying
\[
\sum_{i=1}^{\infty} a_{i} + \sum_{i=1}^{\infty} b_{i} = 1 \; .
\]
Also, let $b_{0} = 0$.  Let the family consist of the discrete
bivariate distributions having density $\pi$ with respect to counting
measure on $\mathbb{N} \times \mathbb{N}$ given by
\[
\pi(x,y) = 
\begin{cases}
a_{x} & x=y, \, y=1,2,3,\ldots \; ;\\
b_{y} & x =y+1, \, y=1,2,3,\ldots\; ; \\
0    & \text{otherwise} \, .
\end{cases}
\]
Hence the marginals are given by
\[
\pi_{X}(x) = \sum_{y=1}^{\infty} \pi(x,y) = \sum_{y=1}^{\infty} a_{x} I(x=y) + b_{y} I(y=x-1) = a_{x} + b_{x-1}
\]
and
\[
\pi_{Y}(y) = \sum_{x=1}^{\infty} \pi(x,y) = \sum_{x=1}^{\infty} a_{x}I(x=y) + b_{y} I(y=x+1) = a_{y} + b_{y} \; .
\]
The full conditionals are easily seen to be
\[
\pi_{X|Y}(x|y) = \frac{a_{y}}{a_{y} + b_{y}} I(x=y) + \frac{b_{y}}{a_{y} + b_{y}} I(x=y+1) \qquad y=1,2,3,\ldots
\]
and
\[
\pi_{Y|X}(y|x) = \frac{a_{x}}{a_{x} + b_{x-1}} I(x=y) + \frac{b_{x-1}}{a_{x} + b_{x-1}} I(y=x-1) \qquad x=1,2,3,\ldots \; .
\]
Define
\[
p_{x} = \frac{a_{x} b_{x}}{(a_{x} + b_{x-1})(a_{x} + b_{x})} \qquad
\text{ and } \qquad q_{x} =   \frac{a_{x-1} b_{x-1}}{(a_{x} +
  b_{x-1})(a_{x-1} + b_{x-1})} \; . 
\]
Then for the DGS 
\[
k_{DGS}(x' , y' | x, y) = \pi_{X|Y}(x' | y) \pi_{Y|X}(y' | x')
\]
and hence for the marginal chain $\Phi_{X}$ 
\[
k_{X}(x' | x) = \sum_{y=1}^{\infty} \pi_{X|Y} (x' | y) \pi_{Y|X}(y | x) = 
\begin{cases}
1 - p_{1} & x' = x = 1 ; \\
p_{x}     & x' = x+ 1, \, x \ge 1 ; \\
q_{x}     & x' = x - 1, \, x \ge 2 ; \\
1- p_{x} - q_{x} & x' = x, \, x \ge 2; \; \text{ and}\\
0        & \text{ otherwise} \; . 
\end{cases}
\]
It is easy to see that the kernel $P_{X}$ is Feller.  If $p \in (0,
1)$, then for the random scan Gibbs sampler (RGS) we have
\[
k_{RGS}(x', y' | x, y) = p \pi_{X|Y}(x' | y) \delta(y' - y) + (1-p)
\pi_{Y|X}(y' | x) \delta(x' - x) \; . 
\]
Since for any open set $O$
\[
P_{RGS} ((x,y), O) = p \sum_{x'=1}^{\infty} \pi_{X|Y}(x' |y) I((x',y) \in
O) + (1-p) \sum_{y'=1}^{\infty} \pi_{Y|X}(y'|x) I((x,y') \in O) 
\]
it is easy to see that $P_{RGS}(\cdot, O)$ is lower
semicontinuous and hence $\Phi_{RGS}$ is Feller.

We are now in position to establish sufficient conditions for the
geometric ergodicity of $\Phi_{X}$, $\Phi_{DGS}$ and $\Phi_{RGS}$.

\begin{lemma}
\label{lem:sufficient for GE}
If
\begin{equation}
\label{eq:sufficient for GE}
\limsup_{x \to \infty} \frac{p_{x}}{q_{x}} < 1 \qquad \text{ and }
\qquad \liminf_{x \to \infty} q_{x} > 0, 
\end{equation}
then $\Phi_{X}$, $\Phi_{DGS}$ and $\Phi_{RGS}$ are geometrically ergodic.
\end{lemma}
\begin{proof}
  We need only verify the conditions of Proposition~\ref{prop:feller drift}
  and we've already seen that both $P_{X}$ and $P_{RGS}$ are Feller.
  Set $V(x) = z^{x}$ for some $z > 1$ which will be determined later
  and note that
\[
G(y) = \sum_{x=1}^{\infty} V(x) \pi_{X|Y}(x|y) = \left( \frac{a_{y} +
    z b_{y}}{a_{y} + b_{y}} \right) z^{y} \; .
\]
For any $d>0$, the sublevel set $A_d:=\{x: V(x) \leq d \}=\{x: z^x
\leq d\}$ is bounded. Since $V$ is a continuous function, $A_d$ is
also closed, hence compact. Therefore $V$ is unbounded off compact
sets on $\sX$. On the other hand, for any $d>0$, the sublevel set
$B_d:=\{y: G(y) \leq d \} \subset \{y: z^{y} \leq d\}$ is
bounded.  Then for any $b > 0$, $W(x,y) = V(x) + b G(y)$ is unbounded
off compact sets on $\sX\times \sY$ because for any $d>0$, $\{(x,y):
W(x,y)\leq d\} \subset A_d \times B_d $ is bounded and closed, hence
compact. Now, all that remains is to construct a drift
  condition for $V$.
 Note that for $x \ge 2$,
\begin{align}\label{eq:V}
  P_{X} V (x) & = \sum_{x'=1}^{\infty} z^{x'} k_{X}(x' | x) \nonumber\\
  & = p_{x} z^{x+1} + q_{x} z^{x-1} + (1 - p_{x} - q_{x}) z^{x} \nonumber\\
  & = \left[z p_{x} + \frac{q_{x}}{z} + 1 - p_{x} - q_{x}\right] z^{x} \nonumber\\
  & = \left[p_{x} (z - 1) + q_{x} \left(\frac{1}{z} - 1 \right)+1
    \right] V(x) \,.
\end{align}
 We next try to bound the coefficient of $V(x)$ in the right hand side
 of \eqref{eq:V} for all large values of $x$. Set 
\[
r : = \limsup_{x \to \infty} \frac{p_{x}}{q_{x}} \qquad \text{ and }
\qquad q := \liminf_{x \to \infty} q_{x} 
\]
and note that $r < 1$ and $q > 0$ by assumption. Then there exists
 $x_{0} \ge 2$ such that
\[
\frac{p_{x}}{q_{x}} < \frac{r+1}{2} \text{\;\;and\;\;} q_x>\frac{q}{2}
\qquad \text{ for all } x > x_{0} \; . 
\]
For any $z \in (1, 2/(r+1))$ and $x > x_{0}$\,,
\begin{align*}
p_{x} (z-1) + q_{x} \left(\frac{1}{z} - 1\right) + 1 & < \frac{r+1}{2} q_{x} (z-1) +
\frac{q_{x} (1-z)}{z} + 1 \\
& = q_{x} (z-1) \left( \frac{r+1}{2} - \frac{1}{z} \right) + 1 \\
& < \frac{q}{2} (z-1) \left( \frac{r+1}{2} - \frac{1}{z} \right) + 1 
\end{align*}
since $z \in (1, 2/(r+1))$ implies
\[
\frac{r+1}{2} - \frac{1}{z}  < 0 \; .
\]
Next note that
\[
 0< q<1\;, \quad 0 < z - 1 < \frac{1-r}{1+r} < 1 \;\;\text{and}\;\;
-\frac{1}{2} < \frac{r+1}{2} - \frac{1}{z} < 0,
\]
which guarantees
\[
0 <  \frac{q}{2} (z-1) \left( \frac{r+1}{2} - \frac{1}{z} \right) +1< 1\,.
\]
Thus there exists $0 < \rho < 1$ such that
\begin{equation}\label{eq:rho}
 \frac{q}{2} (z-1) \left( \frac{r+1}{2} - \frac{1}{z} \right)+ 1 \le \rho < 1 \; .
\end{equation}
Finally, to bound $P_{X} V(x)$ for $x \le x_0$, set 
\begin{equation}\label{eq:L}
L : = \max_{x \le x_{0}} P_{X} V(x)\,. 
\end{equation}
Putting together equations~\eqref{eq:V} to \eqref{eq:L}, we have
\[
P_{X} V(x) \le \rho V(x) + L 
\]
with $0 < \rho < 1$ and $L < \infty$.  The conclusion now follows from
Proposition~\ref{prop:feller drift}.

\end{proof}

The above sufficient condition for geometric ergodicity involves
transition probabilities of the chain $\Phi_X$. Alternatively, we
could state a sufficient condition in terms of the probabilities
$\{a_i, b_i\}$ which define the density $\pi$.

Define 
\[
A  :=
\limsup_{i \to \infty} \frac{a_{i}}{a_{i-1}} \; ; \quad
m  := \liminf_{i \to
  \infty} \frac{a_{i}}{b_{i}} \; ;  \text{ and} \; 
 \quad M  := \limsup_{i \to \infty}
\frac{a_{i}}{b_{i}} \; .
\]

\begin{corollary} \label{cor:sufficient for GE}
If 
\[
\limsup_{i \to \infty} \frac{a_{i}}{b_{i-1}} < \infty, \qquad \qquad
\limsup_{i \to \infty} \frac{b_{i}}{a_{i}} < \infty
\]
and $A(1+M) / (1+m) < 1$, then $\Phi_{X}$, $\Phi_{DGS}$, and
$\Phi_{RGS}$ are geometrically ergodic.
\end{corollary}

\begin{proof}
We verify the conditions of Lemma~\ref{lem:sufficient for GE}.  Note that
\[
q_{i} = \frac{b_{i-1}}{a_{i} + b_{i-1}}
\frac{a_{i-1}}{a_{i-1}+b_{i-1}} = \frac{1}{1 + \frac{a_{i}}{b_{i-1}}}
\frac{1}{1 + \frac{b_{i-1}}{a_{i-1}}} \; .
\]
Hence 
\[
\liminf_{i \to \infty} q_{i} \ge \frac{1}{1 + \limsup_{i \to \infty}
  \, \frac{a_{i}}{b_{i-1}}}
\frac{1}{1 + \limsup_{i \to \infty} \frac{b_{i-1}}{a_{i-1}}}  > 0\; .
\]
Next observe that
\[
\frac{p_{i}}{q_{i}} = \frac{a_{i}}{a_{i-1}} \frac{b_{i}}{b_{i-1}}
\frac{a_{i-1} + b_{i-1}}{a_{i} + b_{i}} = \frac{a_{i}}{a_{i-1}}
\frac{1 + \frac{a_{i-1}}{b_{i-1}}}{1 + \frac{a_{i}}{b_{i}}} \; .
\]
Hence
\[
\limsup_{i \to \infty} \frac{p_{i}}{q_{i}} \le \left[ \limsup_{i \to
    \infty} \frac{a_{i}}{a_{i-1}} \right] \, \left[ \frac{1 +
    \limsup_{i \to \infty} \frac{a_{i-1}}{b_{i-1}}}{1 + \liminf_{i \to
      \infty} \frac{a_{i}}{b_{i}}}  \right] = A \frac{1+M}{1+m} < 1\; .
\]
\end{proof}

So far in this section, we have used Proposition~\ref{prop:feller
  drift} to get sufficient conditions for the geometric ergodicity of
the Markov chains. Next, we use Proposition~\ref{prop:notGE} to study
the conditions under which the Markov chains are subgeometrically
ergodic.

\begin{lemma} \label{lem:sufficient for subGE} 
  The Markov chains $\Phi_{X}$, $\Phi_{DGS}$ and $\Phi_{RGS}$ are subgeometrically ergodic if any one of the following conditions hold:
\begin{enumerate}
\item 
\[
\limsup_{i \to \infty} \frac{\sum_{x=i}^{\infty} (a_{x} +
  b_{x})}{a_{i-1}} = \infty \; ;
\]
\item 
\[
\limsup_{i \to \infty} \frac{\sum_{x=i}^{\infty} (a_{x} +
  b_{x})}{b_{i-1}} = \infty \; ; \text{ or }
\]
\item 
\[
\limsup_{i \to \infty} \frac{b_{i}}{a_{i}} = \infty \; .
\]
\end{enumerate}
\end{lemma}

\begin{proof}
  Let $(X, Y) \sim \varpi$.  For $i=1,2,3,\ldots$ let $H_{i}(x) = I(x
  \ge i)$.  Then
\[
\mu_{i} : = E[ H_{i}(X)] = E[ H_{i}^{2}(X)] = \sum_{x=i}^{\infty}
(a_{x} + b_{x-1}) < \infty 
\]
and
\[
v_{i} : = \text{Var}[H_{i}(X)] = \mu_{i} (1-\mu_{i}) < \infty \; .
\]
Define $h_{i}(x) = [ H_{i}(x) - \mu_{i}] / \sqrt{v_{i}}$ and note that
  $h_{i} \in L_{0,1}^{2}(\varpi_{X})$.  We will show that 
\[
\liminf_{i \to \infty} E[ Var(h_{i}(X) | Y)] = 0 \; ,
\]
and appeal to Proposition~\ref{prop:notGE} for the conclusion.  Let 
\[
\beta_{y} = \frac{b_{y}}{a_{y} + b_{y}} = \pi_{X|Y}(y+1|y) \; .
\]
Then
\begin{align*}
E[ H_{i}(X) | Y=y]& = E[ H_{i}^{2}(X) | Y=y] \\
& = \pi_{X|Y}(y|y) H_{i}(y) + \pi_{X|Y}(y+1|y) H_{i}(y+1) \\
& = \begin{cases}
0 & y \le i-2 ,\\
\beta_{i-1} & y = i-1 ,\\
1 & y \ge i \; .
\end{cases}
\end{align*}
Hence
\[
Var[ H_{i}(X) | Y=y] = 
\begin{cases}
\beta_{i-1}(1 - \beta_{i-1}) & y = i-1 ,\\
0 & \text{otherwise} \; .
\end{cases}
\]
Therefore,
\begin{align*}
E(Var [ H_{i}(X) | Y=y]) & = \sum_{y=1}^{\infty} \pi_{Y}(y) Var [
H_{i}(X) | Y=y] \\
 & = \pi_{Y}(i-1) Var [ H_{i}(X) | Y=i-1] \\
 & = (a_{i-1} + b_{i-1}) \beta_{i-1} (1-\beta_{i-1}) \\
 & = \frac{a_{i-1} b_{i-1}}{a_{i-1} + b_{i-1}} \; .
\end{align*}
Finally,
\[
E [ Var(h_{i}(X) | Y)]= v_{i}^{-1} E [ Var(H_{i}(X) | Y)] =
[\mu_{i}(1-\mu_{i})]^{-1} \frac{a_{i-1} b_{i-1}}{a_{i-1} + b_{i-1}}\; .
\]
Note that
\begin{align*}
(E [ Var(h_{i}(X) | Y)])^{-1} & = \mu_{i}(1-\mu_{i}) \frac{a_{i-1} +
  b_{i-1}}{a_{i-1} b_{i-1}}  \\
& = (1-\mu_{i}) \left[ \sum_{x=i}^{\infty} (a_{x} + b_{x}) + b_{i-1}
\right] \left( \frac{1}{a_{i-1}} + \frac{1}{b_{i-1}}\right)  \\
& = (1-\mu_{i}) \left[ \frac{\sum_{x=i}^{\infty} (a_{x} +
    b_{x})}{a_{i-1}} + \frac{\sum_{x=i}^{\infty} (a_{x} +
    b_{x})}{b_{i-1}} + \frac{b_{i-1}}{a_{i-1}} + 1 \right]
\end{align*}
and that
\[
\lim_{i \to \infty} ( 1- \mu_{i}) = \lim_{i \to \infty}
\sum_{x=1}^{i-1} (a_{x} + b_{x-1}) = 1 \; .
\]
Hence equation \eqref{eq:liminf} holds if and only if 
\begin{align*}
& 
\limsup_{i \to \infty} \frac{\sum_{x=i}^{\infty} (a_{x} +
  b_{x})}{a_{i-1}} = \infty , \\
\text{or} \quad & \limsup_{i \to \infty} \frac{\sum_{x=i}^{\infty} (a_{x} +
  b_{x})}{b_{i-1}} = \infty, \\
\text{or} \quad & \limsup_{i \to \infty} \frac{b_{i}}{a_{i}} = \infty \; .
\end{align*}
\end{proof}

Finally, we can use the previous results to characterize the
conditions for geometric ergodicity of TGS Markov chains for a large
subfamily of our discrete distributions.

\begin{corollary} \label{cor:subfamily}
Assume that both $A : = \lim_{i \to \infty} \frac{a_{i}}{a_{i-1}}$ and
$\lim_{i \to \infty} \frac{a_{i}}{b_{i}}$ exist.  Then all the limits
  below are well defined and the following statements are equivalent:
\begin{enumerate}
\item[(a)] 
\[
\lim_{i \to \infty} \frac{a_{i}}{b_{i-1}} < \infty, \quad \lim_{i \to
  \infty} \frac{b_{i}}{a_{i}} < \infty, \quad \text{ and } \quad A < 1
\, .
\]
\item[(b)]
\[
r = \lim_{i \to \infty} \frac{p_{i}}{q_{i}} < 1 \quad \text{ and }
\quad q = \lim_{i \to \infty} q_{i} > 0 \, .
\]
\item[(c)] $\Phi_{X}$ is geometrically ergodic.
\item[(d)] $\Phi_{DGS}$ is geometrically ergodic.
\item[(e)] $\Phi_{RGS}$ is geometrically ergodic. 
\end{enumerate}
\end{corollary}

\begin{proof}
As we noted in Section~\ref{sec:three chains}, the equivalence of (c) and (d) is well known.

\noindent\underline{(a) $\Rightarrow$ (b)}: Note that 
\[
q = \lim_{i \to \infty}   q_{i} = \frac{1}{1 + \lim_{i \to \infty} \frac{a_{i}}{b_{i-1}}}
\frac{1}{1 + \lim_{i \to \infty} \frac{b_{i-1}}{a_{i-1}}} > 0 
\]
and
\[
r = \lim_{i \to \infty} \frac{p_{i}}{q_{i}} = \left[ \lim_{i \to
    \infty}\frac{a_{i}}{a_{i-1}} \right] \, \left[ \frac{1 + \lim_{i
      \to \infty} \frac{a_{i-1}}{b_{i-1}}}{1 + \lim_{i \to \infty}
    \frac{a_{i}}{b_{i}}} \right] = A < 1\; .
\]

\noindent\underline{(b) $\Rightarrow$ (c) and (b) $\Rightarrow$ (e)}:
The same argument holds for $\Phi_{X}$ and $\Phi_{RGS}$. Immediate by
Lemma~\ref{lem:sufficient for GE}.\\ 

\noindent\underline{(c) $\Rightarrow$ (a) and (e) $\Rightarrow$ (a)}:
The same argument holds for $\Phi_{X}$ and $\Phi_{RGS}$. If the chain
is geometrically ergodic, then $\lim_{i \to \infty}
\frac{a_{i}}{b_{i-1}} < \infty$ and $\lim_{i \to \infty}
\frac{b_{i}}{a_{i}} < \infty$ by conditions 2 and 3 of
Lemma~\ref{lem:sufficient for subGE}.  Next, if $A=1$, then for any
fixed positive integer $K$, we have
\[
\lim_{i \to \infty} \frac{a_{i+1}}{a_{i}}= 1, \; \lim_{i \to \infty}
\frac{a_{i+2}}{a_{i}}= 1, \ldots, \lim_{i \to \infty}
\frac{a_{i+K}}{a_{i}} =1 \; .
\]
Then there exists $i_{0}$ such that for any $i \ge i_{0}$,
\[
\frac{a_{i+1}}{a_{i}} > \frac{1}{2}, \;  \frac{a_{i+2}}{a_{i}} >
\frac{1}{2}, \ldots, \frac{a_{i+K}}{a_{i}} > \frac{1}{2}\; .
\]
Hence, given any $K$, there exists $i_{0}$ such that for any $i >
i_{0}$,
\[
\frac{\sum_{x=i}^{\infty} (a_{x} + b_{x})}{a_{i-1}}\ge
\frac{\sum_{x=i}^{i + K - 1} a_{x}}{a_{i-1}} >\frac{K}{2} 
\]
which implies 
\[
\limsup_{i \to \infty} \frac{\sum_{x=i}^{\infty} (a_{x} +
  b_{x})}{a_{i-1}} = \infty \; .
\]
Thus by condition 1 of Lemma~\ref{lem:sufficient for subGE}, the
chains are subgeometrically ergodic--a contradiction of (c).  So $A
\neq 1$.  But $A$ cannot be greater than 1 either since otherwise
$\sum_{x=1}^{\infty} a_{x} = \infty$ which contradicts the fact that
$\sum_{x=1}^{\infty} a_{x} + \sum_{x=1}^{\infty} b_{x} = 1$.
Therefore, $A < 1$.
\end{proof}

To better understand the conditions for geometric ergodicity provided
in Corollary~\ref{cor:subfamily}, we hereby explain its condition~(a)
explicitly. First, the requirement that $A=\lim_{i \to \infty}
\frac{a_{i}}{a_{i-1}}<1$ implies that for any $0<A_1<A<A_2<1$, there
exists $i_0$ such that for any $i>i_0$, $a_i/a_{i-1}\in(A_1,A_2)$,
hence $a_i\in (a_{i_0}A_1^{i-i_0},a_{i_0}A_2^{i-i_0})$. In other
words, the sequence $\{a_i\}$ decays at a geometric rate as $i$
increases. Secondly, the requirements $\lim_{i \to \infty}
\frac{a_{i}}{b_{i-1}} < \infty$ and $\lim_{i \to \infty}
\frac{b_{i}}{a_{i}} < \infty$ imply that there exist $0<B_1,
B_2<\infty$ such that, for any $i>i_0$, $a_{i+1}/b_{i}<B_1$ and
$b_{i}/a_{i}<B_2$, hence $b_i \in (a_{i+1}B_1 ,a_i B_2) \subset (a_i
A_1B_1, a_i B_2)$. That is, $b_i=O(a_i)$ as $i \to \infty$. In
summary, Condition~(a) requires that the sequences $\{a_i\}$ and
$\{b_i\}$ both decay geometrically at the same rate as $i$ increases.

  We close this section by considering four concrete examples.

\begin{example}
Let $a_{x} = c_{1} x^{-d}$ and $b_{x} = c_{2} x^{-d}$ where $d > 1$
and $(c_{1} + c_{2}) \sum_{x=1}^{\infty} x^{-d} = 1$.  Then both
$\lim_{i \to \infty} \frac{a_{i}}{a_{i-1}}$ and $\lim_{i \to \infty}
\frac{a_{i}}{b_{i}}$ exist, with $A= 1$.  Therefore, $\Phi_{X}$,
$\Phi_{DGS}$ and $\Phi_{RGS}$ are subgeometrically ergodic by
Corollary~\ref{cor:subfamily}. 
\end{example}

\begin{example}
  Let $c$ satisfy $(1+c)e^{-1} / (1 - e^{-1}) = 1$.  Set $a_{x} = c
  e^{-x}$ and $b_{x} = e^{-x}$.  Then both $\lim_{i \to \infty}
  \frac{a_{i}}{a_{i-1}}$ and $\lim_{i \to \infty} \frac{a_{i}}{b_{i}}$
  exist, with $A= e^{-1} < 1$.  Furthermore, $\limsup_{i \to \infty}
  \frac{a_{i}}{b_{i-1}} = \lim_{i \to \infty} c e^{-1} < \infty$ and
  $\limsup_{i \to \infty} \frac{b_{i}}{a_{i}} = c^{-1} < \infty$.
  Therefore, $\Phi_{X}$, $\Phi_{DGS}$ and $\Phi_{RGS}$ are all
  geometrically ergodic by Corollary~\ref{cor:subfamily}.
\end{example}

\begin{example}
Let $c$ satisfy $c e^{-1} / (1 - e^{-1}) + e^{-2} / (1 - e^{-2})=1$.
Set $a_{x} = ce^{-x}$ and $b_{x} = e^{-2x}$.  Then both $\lim_{i \to
  \infty} \frac{a_{i}}{a_{i-1}}$ and $\lim_{i \to \infty}
\frac{a_{i}}{b_{i}}$ exist.  Also, 
\[
\limsup_{i \to \infty} \frac{a_{i}}{b_{i-1}} = \lim_{i \to \infty} c
e^{i-2} = \infty \; .
\] 
Therefore, $\Phi_{X}$, $\Phi_{DGS}$ and $\Phi_{RGS}$ are subgeometrically ergodic by Corollary~\ref{cor:subfamily}.
\end{example}

\begin{example}
Let $c$ satisfy $c e^{-1} / (1 - e^{-1}) + e^{-2} (1 - e^{-2}) = 1$.
Set
\[
a_{x} = \begin{cases}
ce^{-x} & x \text{ even } \\
e^{-2x} & x \text{ odd }
\end{cases}
\qquad \text{and} \qquad
b_{x} = \begin{cases}
e^{-2x} & x \text{ even }\\
ce^{-x} & x \text{ odd } 
\end{cases} \; .
\]
Then $\lim_{i \to \infty} \frac{a_{i}}{b_{i}}$ does not exist. Hence
Corollary~\ref{cor:subfamily} is not applicable.  Instead we have to
use Lemma~\ref{lem:sufficient for subGE}.  Notice that
\[
\limsup_{i \to \infty} \frac{b_{i}}{a_{i}} \ge \lim_{i \to \infty}
\frac{b_{2i+1}}{a_{2i+1}} = \lim_{i \to \infty}
\frac{ce^{-(2i+1)}}{e^{-2(2i+1)}} = \lim_{i \to \infty} c e^{2i+1} =
\infty
\]
 and hence $\Phi_{X}$, $\Phi_{DGS}$ and $\Phi_{RGS}$ are subgeometrically
ergodic. 
\end{example}

\bibliographystyle{imsart-nameyear}
\bibliography{mcref}
\end{document}